\documentclass{amsart}

\usepackage{amssymb,amsmath,amsthm,amsaddr}
\usepackage{tikz,comment,enumerate,hyperref}

\newcommand{\alg}{\mathbf}
\newcommand{\class}{\mathsf}
\newcommand{\logic}{\mathrm}

\newcommand{\pair}[2]{\langle #1, #2 \rangle}
\newcommand{\triple}[3]{\langle #1, #2, #3 \rangle}
\newcommand{\quadruple}[4]{\langle #1, #2, #3, #4 \rangle}

\newcommand{\set}[2]{\{ #1 \mid #2 \}}
\newcommand{\assign}{:=}
\newcommand{\dmneg}{{-}}

\newcommand{\iso}{\cong}
\newcommand{\Rrel}{R}
\newcommand{\inequals}{\leq}

\DeclareMathOperator{\Con}{Con}
\newcommand{\Leibniz}[2]{\Omega^{#1} #2}

\newcommand{\True}{\mathsf{t}}
\newcommand{\False}{\mathsf{f}}
\newcommand{\Neither}{\mathsf{n}}
\newcommand{\Both}{\mathsf{b}}
\newcommand{\Inter}{\mathsf{i}}

\newcommand{\IsTrue}{\mathrm{True}}
\newcommand{\NonFalse}{\mathrm{NonFalse}}
\newcommand{\ExTrue}{\mathrm{ExTrue}}

\newcommand{\Btwo}{\alg{B}_{\alg{2}}}
\newcommand{\Kthree}{\alg{K}_{\alg{3}}}
\newcommand{\DMfour}{\alg{DM}_{\alg{4}}}

\newcommand{\K}{\logic{K}}
\newcommand{\LP}{\logic{LP}}
\newcommand{\BD}{\logic{BD}}
\newcommand{\ETL}{\logic{ETL}}

\newtheorem{theorem}{Theorem}[section]

\newtheorem{lemma}[theorem]{Lemma}
\newtheorem{fact}[theorem]{Fact}
 
\theoremstyle{definition}
\newtheorem{definition}[theorem]{Definition}

\author{Adam P\v{r}enosil}
\title[Four-valued logics of truth, nonfalsity, \dots]{Four-valued logics of truth, non-falsity, exact~truth, and material equivalence}
\address{Vanderbilt University, Nashville, TN, USA}
\email{adam.prenosil@vanderbilt.edu}
\date{}
\thanks{This research was supported by grant 17-18344Y of the Czech Science Foundation. The~author is grateful to an anonymous referee for a thorough reading of the manuscript and helpful suggestions which served to improve the paper.}

\begin{document}

\begin{abstract}
  The four-valued semantics of Belnap--Dunn logic, consisting of the truth values True, False, Neither, and Both, gives rise to several non-classical \mbox{logics} depending on which feature of propositions we wish to preserve: truth, non-falsity, or exact truth (truth and non-falsity). Interpreting equality of truth values in this semantics as material equivalence of propositions, we can moreover see the equational consequence relation of this four-element algebra as a logic of material equivalence. In this paper we \mbox{axiomatize} all combinations of these four-valued logics, for example the logic of truth and exact truth or the logic of truth and material equivalence. These combined systems are \mbox{consequence} \mbox{relations} which allow us to express implications involving more than one of these features of propositions.
\end{abstract}

\maketitle

\section{Introduction}
\label{sec: intro}

  The well-known four-valued logic introduced by Belnap and Dunn~\cite{belnap77a,belnap77b,dunn76} in the 1970's is based on the simple idea that truth and falsehood may be treated as independent rather than complementary features of propositions. This yields a four-valued semantics whose four truth values are True, False, Neither (True~nor~False), and Both (True and False), abbreviated as $\True$, $\False$, $\Neither$, and $\Both$. The logic itself is then defined in terms of truth-preservation (or non-falsity-preservation).

  The four-valued algebra $\DMfour$ of Belnap and Dunn was already known at the time to generate the variety of De~Morgan lattices as a quasivariety (see \cite{kalman58}). That is, the equational consequence of $\DMfour$ can be axiomatized by adding the De \mbox{Morgan} laws for conjunction and disjunction and the law of double negation introduction and elimination to the axioms of distributive lattices. Giving the equality relation a logical reading as material equivalence of propositions (equality of truth values), this yields an axiomatization of the logic of four-valued material equivalence.

  More recently, Pietz and Rivieccio~\cite{pietz+rivieccio13,rivieccio12} realized that the same four truth values also yield what they call Exactly True Logic if instead of preserving truth we wish to preserve exact truth (truth and non-falsity). The~same idea was previously considered by Marcos~\cite{marcos11} for a logic with a different selection of connectives.

  The basic premise of the present paper is now straightforward: consider the unary predicates $\IsTrue$, $\ExTrue$ (``exactly true''), and $\NonFalse$ interpreted in the four-valued semantics of Belnap and Dunn by the sets $\{ \True, \Both \}$, $\{ \True \}$, and $\{ \True, \Neither \}$ \mbox{respectively}. We~\mbox{axiomatize} all \mbox{logics} (single-conclusion or multiple-conclusion con\-sequence re\-lations) determined by this four-valued semantics which allow for the use of more than one of the above predicates.

  To~take a simple example, the following inferences are valid in this semantics:
\begin{align*}
  \ExTrue(x) ~ \& ~ \IsTrue(\dmneg x \vee y) & \vdash \IsTrue(y), \\
  \IsTrue(x) ~ \& ~ \NonFalse(\dmneg x \vee y) & \vdash \NonFalse(y).
\end{align*}
  More explicitly, the first rule states that if $a \in \{ \True \}$ and $\dmneg a \vee b \in \{ \True, \Both \}$, then $b \in \{ \True, \Both \}$ for each $a, b \in \DMfour$. Similarly, if the binary predicate $\approx$ is interpreted by equality of truth values, then the following inference is also valid:
\begin{align*}
  \IsTrue(x) ~ \& ~ \IsTrue(y) \vdash \dmneg x \vee y \approx y.
\end{align*}
  We therefore see that there is some non-trivial interaction between the predicates $\ExTrue$, $\NonFalse$, $\IsTrue$, and $\approx$. The present paper is devoted to axiomatizing this inter\-action between all combinations of these predicates.\footnote{A different approach to combining Belnap--Dunn logic with the Exactly True Logic (as well as some other related logics) was taken by Shramko et al.~\cite{shramko+zaitsev+belikov19}, resulting in what they call bi-consequence systems.}

  We also consider what happens if we expand the signature of these logics by constants representing some of the elements $\True$, $\False$, $\Neither$, and $\Both$. Such an expansion will trivialize the problem in some cases. For example, although the predicate $\IsTrue$ is not definable in terms of the equality predicate in the basic signature, it becomes definable by the equation $\Both \wedge x \approx \Both$ if the constant $\Both$ is available.

  One motivation for expanding the relational signature of Belnap--Dunn logic comes from Belnap's original proposal~\cite{belnap77a} to use four-valued logic to handle information collected from a variety of sources which may be mutually inconstent as well as incomplete. Each such source informs us that certain propositions are true and certain others are false, and inconsistency easily arises when two sources are in conflict. Belnap's original relational signature, consisting merely of the predicate $\IsTrue$, is well chosen in that the input from our sources must take this form if we are to pool input from inconsistent sources directly. If one source claims that a proposition is true and another claims that it is false, we simply assign the value $\Both$ to this proposition. But if one source claims that a proposition is not false and another claims that it is, we are in as much trouble as in classical logic.

  The fact that the relational signature of Belnap--Dunn logic is appropriate for recording the input from our sources, however, does not imply that we must restrict ourselves to this limited relational signature when reasoning \emph{about} the information collected in this way. Indeed, the user of such a database of potentially inconsistent information is surely interested in differentiating between cases where we only have information supporting a proposition $p$ (in Belnap's terminology, $p$ is just told True) and cases where we have information supporting both $p$ and its negation (in Belnap's terminology, $p$ is told both True and False). In the conclusion of a logical inference, this distinction enables us to make better-informed decisions, while in the premises it enables us to make full use the available information.

  We now outline the structure of this paper. In Section~\ref{sec: preliminaries} we review the framework of abstract algebraic logic adopted in the present paper and we recall some known facts about De~Morgan lattices and Belnap--Dunn logic. In Section~\ref{sec: without equality} we axiomatize the logics of all combinations of the predicates $\IsTrue$, $\ExTrue$, and $\NonFalse$. This section of the paper is largely based on the author's thesis~\cite[Ch.~10]{prenosil18}. In Section~\ref{sec: with equality} we extend these logics by the equality predicate. Finally, we axiomatize the multiple-conclusion versions of these logics in Section~\ref{sec: mc}.

\section{Preliminaries}
\label{sec: preliminaries}

  In this preliminary section, we review the basic concepts of abstract algebraic logic. We also recall some facts about De~Morgan lattices and Belnap--Dunn logic. For~a more thorough introduction to abstract algebraic logic, the reader may consult the textbook~\cite{font16}. Although researchers in abstract algebraic logic generally assume that the relational signature of a logic consists of a single unary predicate, this is not a substantial feature of the theory and indeed we abandon this assumption here. The~theory for more general relational signatures was developed in \cite{dellunde+jansana96,elgueta94thesis}. We~shall assume \mbox{familiarity} with the basic notions of universal algebra, such as the reader might obtain from the textbook~\cite{burris+sankappanavar81}.

\subsection{Logics and structures}
\label{subsec: logics}

  We first define our syntactic objects of interest, namely formulas and rules. A \emph{signature} consists of a set of relational symbols of finite arity (called a \emph{relational signature}), a set of function symbols of finite arity (called an \emph{algebraic signature}), and an infinite set of variables. A~\emph{term} in a given signature is an element of the absolutely free algebra in the algebraic signature over the given set of variables. A~\emph{formula} in a given signature consists of a relational symbol $\Rrel$ of arity $n$ in the relational signature and an $n$-tuple of terms $t_{1}, \dots, t_{n}$ in the algebraic signature, written as $\Rrel(t_{1}, \dots, t_{n})$. A \emph{rule} is then a pair consisting of a set of formulas $\Gamma$ and a formula $\varphi$, written as $\Gamma \vdash \varphi$. The elements of the set $\Gamma$ will be called the \emph{premises} of the rule and the formula $\varphi$ will be called the \emph{conclusion} of the rule. Rules with a finite set of premises $\Gamma = \{ \gamma_{1}, \dots, \gamma_{n} \}$ will be written as $\gamma_{1} ~ \& ~ \dots ~ \& ~ \gamma_{n} \vdash \varphi$. For~example, $\IsTrue(x) ~ \& ~ \IsTrue(y) \vdash x \approx y$ is a rule in the relational signature which consists of a unary predicate $\IsTrue$ and a binary predicate ${\approx}$. We emphasize that the predicate ${\approx}$ need not be interpreted by the equality relation.
  
  Let us now consider a certain fixed signature. An \emph{algebra} $\alg{A}$ consists of a set $A$ and a \mbox{function} $f\colon A^{n} \to A$ for each function symbol of arity $n$. A \emph{structure} then consists of an \mbox{algebra} $\alg{A}$ and a relation $\Rrel \subseteq A^{n}$ for each relational symbol of arity~$n$. We~generally do not distinguish in notation between functions (relations) and function symbols (relational symbols). A \emph{valuation} on a structure is a homo\-morphism from the absolutely free algebra (the algebra of terms) to the algebra $\alg{A}$. A valuation $v$ is said to \emph{validate} a formula $\Rrel(t_{1}(x_{1}, \dots, x_{m}), \dots, t_{n}(x_{1}, \dots, x_{m}))$ if $\langle t_{1}(a_{1}, \dots, a_{m}), \dots, t_{n}(a_{1}, \dots, a_{m}) \rangle \in \Rrel$ where $a_{i} = v(x_{i})$ for each $x_{i}$. A rule is \emph{valid} or \emph{holds} in a structure if each valuation on the structure which validates the premises of the rule also validates its conclusion.

 For example, the rule $\IsTrue(x) ~ \& ~ \IsTrue(y) \vdash x \approx y$ holds in the two-element \mbox{structure} with the universe $\{ \True, \False \}$ where the predicate $\IsTrue$ is interpreted by the singleton set $\{ \True \}$ and the predicate $\approx$ by the equality relation, but it fails in the three-element structure with the universe $\{ \True, \Both, \False \}$ where $\IsTrue$ is interpreted by the set $\{ \True, \Both \}$ and $\approx$ is interpreted by the equality relation.

  The rules defined above are nothing but strict universal Horn sentences without equality in (barely noticeable) disguise, at least if $\Gamma$ is finite. Recall that a \emph{strict Horn formula (without equality)} is a disjunction of finitely many negated atomic formulas and a single atomic formula, consisting of a relational symbol of arity $n$ applied to an $n$-tuple of terms. We do not assume here the \mbox{presence} of a privileged binary relational symbol designated to represent the \mbox{equality} \mbox{relation}: the~binary predicate $\approx$ may be interpreted by any binary relation. A~\emph{strict universal Horn \mbox{sentence} (without equality)} then arises from such a formula by universally quantifying over all variables occurring in the formula. The rule $\IsTrue(x) ~ \& ~ \IsTrue(y) \vdash x \approx y$ thus corresponds to the strict universal Horn sentence (without equality)
\begin{align*}
  \forall x \, \forall y \, (\neg \IsTrue(x) \vee \neg \IsTrue(y) \vee x \approx y).
\end{align*}

  The \emph{logic defined by a structure} is the set of all rules which hold in the structure. The \emph{logic defined by a class of structures} $\class{K}$ is the set of all rules which hold in all structures in $\class{K}$. A \emph{logic (simpliciter)} is a logic defined by some class of structures. (Logics can also be defined intrinsically as sets of rules which satisfy certain closure conditions.) A \emph{model} of a set of rules is a structure which validates all of these rules. Finally, we say that a logic $\logic{L}$ is \emph{axiomatized} by a set of rules if $\logic{L}$ is the smallest logic which contains these rules. Equivalently, $\logic{L}$ is axiomatized by a set of rules if and only if the models of $\logic{L}$ are precisely the structures which validate these rules. A rule \emph{holds (fails)} it a logic if it belongs (does not belong) to the logic.

\begin{fact}
  The class of all models of a logic $\logic{L}$ is closed under substructures and products of structures.
\end{fact}

\begin{fact}\label{fact: intersection of filters}
  If the structures $\langle \alg{A}, R^{i}_{1}, \dots, R^{i}_{m} \rangle$ are models of a logic $\logic{L}$ for each $i \in I$, then so is the structure $\langle \alg{A}, R_{1}, \dots, R_{m} \rangle$ for $R_{k} \assign \bigcap_{i \in I} R^{i}_{k}$.
\end{fact}

  The structure $\langle \alg{A}, R_{1}, \dots, R_{m} \rangle$ is \emph{trivial} if each $R_{i}$ is the total relation on $\alg{A}$ (if $R_{i} = \alg{A}^{n}$ where $n$ is the arity of $R_{i}$). Every rule holds in a trivial structure.

\subsection{Leibniz reducts}
\label{subsec: leibniz}

  For each set $F \subseteq \alg{A}$ we define the congruence $\Leibniz{\alg{A}}{F} \in \Con \alg{A}$, called the \emph{Leibniz congruence} of $F$, as follows: $\pair{a}{b} \in \Leibniz{\alg{A}}{F}$ if and only if
\begin{align*}
  p(a) \in F \iff p(b) \in F \text { for each polynomial } p \text{ on } \alg{A}.
\end{align*}
  Here a \emph{polynomial} on $\alg{A}$ is a function $p\colon \alg{A} \to \alg{A}$ such that $p(a) = t(a, c_{1}, \dots, c_{n})$ for some term $t(x, z_{1}, \dots, z_{n})$ and some $c_{1}, \dots, c_{n} \in \alg{A}$. It~will be useful to observe that for each family of sets $F_{i} \subseteq \alg{A}$ with $i \in I$ we have
  \begin{align*}
    \bigcap_{i \in I} (\Leibniz{\alg{A}}{F_{i}}) \subseteq \Leibniz{\alg{A}}{\bigcap_{i \in I} F_{i}}.
  \end{align*}
  
    Likewise, for each binary relation $\Rrel \subseteq \alg{A}^{2}$ we define the Leibniz congruence $\Leibniz{\alg{A}}{\Rrel}$ as follows: $\pair{a}{b} \in \Leibniz{\alg{A}}{\Rrel}$ if and only if
\begin{align*}
  \pair{p(a)}{q(a)} \in \Rrel \iff \pair{p(b)}{q(b)} \in \Rrel \text { for all polynomials } p, q \text{ on } \alg{A}.
\end{align*}
  This definition extends in an obvious way to relations of arbitrary finite arity, but in this paper we restrict our attention to structures with unary and binary relations.

  Alternatively, $\Leibniz{\alg{A}}{F}$ may be defined as the largest congruence $\theta \in \Con \alg{A}$ which is compatible with $F$ in the sense that
  \begin{align*}
    a \in F ~ \& ~ \pair{a}{b} \in \theta \implies b \in F.
  \end{align*}
  In particular, if $\theta$ is compatible with $F$, then we define
  \begin{align*}
    F / \theta \assign \set{[a]_{\theta} \in \alg{A} / \theta}{a \in F}.
  \end{align*}
  Similarly, if $\Rrel$ is a binary relation, then $\Leibniz{\alg{A}}{\Rrel}$ may be defined as the largest congruence $\theta \in \Con \alg{A}$ which is compatible with $\Rrel$ in the sense that
  \begin{align*}
    \pair{a}{b} \in \Rrel ~ \& ~ \pair{a}{c}, \pair{b}{d} \in \theta \implies \pair{c}{d} \in \Rrel.
  \end{align*}
  In particular, if $\theta$ is compatible with $\Rrel$, then we define
  \begin{align*}
    \Rrel / \theta \assign \set{\pair{[a]_{\theta}}{[b]_{\theta}}}{\pair{a}{b} \in \Rrel}.
  \end{align*}

\begin{definition}
  The \emph{Leibniz congruence} of a structure $\langle \alg{A}, \Rrel_{1}, \dots, \Rrel_{n} \rangle$ is the congruence $\theta_{L} \assign \Leibniz{\alg{A}}{\Rrel_{1}} \cap \dots \cap \Leibniz{\alg{A}}{\Rrel_{n}}$. The structure is called \emph{reduced} if the Leibniz congruence is the identity relation on $\alg{A}$. The \emph{Leibniz reduct} of $\langle \alg{A}, \Rrel_{1}, \dots, \Rrel_{n} \rangle$ is the structure $\langle \alg{A} / \theta_{L}, \Rrel_{1} / \theta_{L}, \dots, \Rrel_{n} / \theta_{L} \rangle$.
\end{definition}

\begin{fact}\label{fact: leibniz}
  Let $\theta_{L}$ be the Leibniz congruence of $\langle \alg{A}, \Rrel_{1}, \dots, \Rrel_{n} \rangle$ and let $\theta$ be a congruence on $\alg{A}$ such that $\theta \subseteq \theta_{L}$. Then $\langle \alg{A}, \Rrel_{1}, \dots, \Rrel_{n} \rangle$ is a model of a logic $\logic{L}$ if and only if the structure $\langle \alg{A} / \theta, \Rrel_{1} / \theta, \dots, \Rrel_{n} / \theta \rangle$ is a model of $\logic{L}$. Moreover, $\langle \alg{A} / \theta_{L}, \Rrel_{1} / \theta_{L}, \dots, \Rrel_{n} / \theta_{L} \rangle$ is a reduced structure.
\end{fact}

  In particular, a rule holds in $\logic{L}$ if and only if it holds in each \emph{reduced} model of~$\logic{L}$.

\subsection{De~Morgan lattices}
\label{subsec: dmls}

  We now review some basic facts about De~Morgan lattices. The proofs of these facts can be found in~\cite{pynko99a}.

\begin{definition}[De Morgan lattices]
  A \emph{De~Morgan \mbox{lattice}} is a distributive~lattice with an order-inverting involution, denoted here by $\dmneg x$. That is, De~Morgan \mbox{lattices} form a variety axiomatized by the equations $\dmneg \dmneg x \approx x$ and $\dmneg (x \wedge y) \approx \dmneg x \vee \dmneg y$, or equivalently $\dmneg \dmneg x \approx x$ and $\dmneg (x \vee y) \approx \dmneg x \wedge \dmneg y$. A \emph{Kleene lattice} is a De~Morgan lattice which moreover satisfies $x \wedge \dmneg x \leq y \vee \dmneg y$.
\end{definition}

  Since De~Morgan lattices are expansions of distributive lattices, the following well-known lemma applies to them. Here an ideal $I$ is prime if $a \wedge b \in I$ implies $a \in I$ or $b \in I$, and a filter $F$ is prime if $a \vee b \in F$ implies $a \in F$ or $b \in F$. We~do not require that prime ideals and filters be proper or non-empty.

\begin{lemma}[Pair Extension Lemma]\label{lemma: filter-ideal separation}
  Let $F$ be a filter and $I$ be an ideal on a distributive lattice $\alg{L}$ such that $F \cap I = \emptyset$. Then there is a prime filter $G \supseteq F$ and a prime ideal $J \supseteq I$ such that $G \cap J = \emptyset$.
\end{lemma}

  The three most important De~Morgan lattices are depicted in Figure~\ref{fig: si dmas}. In~these algebras the De~Morgan negation corresponds to reflection across the horizontal axis of symmetry. In particular, the De~Morgan negation of $\Kthree$ has one fixpoint and the De~Morgan negation of $\DMfour$ has two fixpoints. Note that $\DMfour$ is not a Kleene lattice but $\Btwo$ and $\Kthree$ are. The algebra $\Kthree$ may be embedded into~$\DMfour$ in two different ways, namely $\Inter \mapsto \Neither$ and $\Inter \mapsto \Both$. In an abuse of notation, we let $\Kthree$ stand for either of the three isomorphic algebras with universes $\{ \False, \Inter, \True \}$, $\{ \False, \Neither, \True \}$, $\{ \False, \Both, \True \}$.

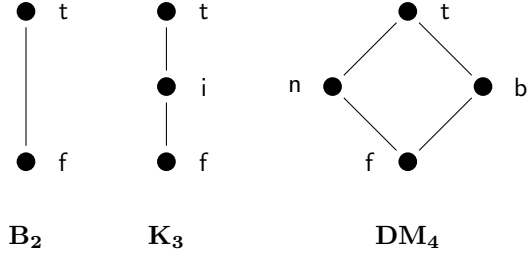
\begin{figure}
\caption{Subdirectly irreducible De~Morgan lattices}
\label{fig: si dmas}

\begin{center}
\begin{tikzpicture}[scale=1,
  dot/.style={circle,fill,inner sep=2.5pt,outer sep=2.5pt}]
  \node (B2a) at (0,0) [dot] {};
  \node (B2b) at (0,2) [dot] {};
  \node at (0.5,0) {$\False$};
  \node at (0.5,2) {$\True$};
  \draw[-] (B2a) -- (B2b);
  \node at (0,-1) {$\Btwo$};
\end{tikzpicture}
\qquad
\begin{tikzpicture}[scale=1,
  dot/.style={circle,fill,inner sep=2.5pt,outer sep=2.5pt}]
  \node (K3a) at (0,0) [dot] {};
  \node (K3b) at (0,1) [dot] {};
  \node (K3c) at (0,2) [dot] {};
  \node at (0.5,0) {$\False$};
  \node at (0.5,1) {$\Inter$};
  \node at (0.5,2) {$\True$};
  \draw[-] (K3a) -- (K3b) -- (K3c);
  \node at (0,-1) {$\Kthree$};
\end{tikzpicture}
\qquad
\begin{tikzpicture}[scale=1,
  dot/.style={circle,fill,inner sep=2.5pt,outer sep=2.5pt}]
  \node (DM4a) at (0,0) [dot] {};
  \node (DM4b) at (-1,1) [dot] {};
  \node (DM4c) at (1,1) [dot] {};
  \node (DM4d) at (0,2) [dot] {};
  \node at (0.5,2) {$\True$};
  \node at (-0.5,0) {$\False$};
  \node at (1.5,1) {$\Both$};
  \node at (-1.5,1) {$\Neither$};
  \draw[-] (DM4a) edge (DM4b);
  \draw[-] (DM4a) edge (DM4c);
  \draw[-] (DM4b) edge (DM4d);
  \draw[-] (DM4c) edge (DM4d);
  \node at (0,-1) {$\DMfour$};
\end{tikzpicture}
\end{center}
\end{figure}

\begin{fact}\label{fact: hf}
  Let $\alg{A}$ be a De~Morgan lattice and $T$ be a prime filter on $\alg{A}$. Then the map $h_{T}\colon \alg{A} \to \DMfour$ defined as
  \begin{align*}
    h_{T}(a) & \assign \True \text{ if } a \in T \text{ and } \dmneg a \notin T, \\
    h_{T}(a) & \assign \Both \text{ if } a \in T \text{ and } \dmneg a \in T, \\
    h_{T}(a) & \assign \Neither \text{ if } a \notin T \text{ and } \dmneg a \notin T, \\
    h_{T}(a) & \assign \False \text{ if } a \notin T \text{ and } \dmneg a \in T,
  \end{align*}
  is a homomorphism of De~Morgan lattices.
\end{fact}

  This observation combined with the Pair Extension Lemma yields the following.

\begin{fact}\label{fact: si dmas}
  Each De~Morgan lattice is a subdirect product of $\Btwo$, $\Kthree$, and $\DMfour$.
\end{fact}

  In particular, the variety of De~Morgan lattices is generated as a quasivariety by~$\DMfour$. Now consider instead of the De~Morgan lattice $\DMfour$ its expansion by some of the constants $\True, \Neither, \Both$, interpreted of course by the corresponding elements of~$\DMfour$. (We can disregard the constant $\False$, since $\False = \dmneg \True$.) We~shall use $\True$ to denote the top element of a De~Morgan lattice whenever it exists.

  To axiomatize the quasivariety generated by such an expansion of $\DMfour$, it suffices to add all of the applicable equations from the following list:
  \begin{align*}
    \True & \approx x \vee \True, &
    \Neither & \approx \dmneg \Neither, &
    \Both & \approx \dmneg \Both, &
    \Neither \vee \Both & \approx x \vee \Neither \vee \Both.
  \end{align*}
  These equations ensure that in each factor in a subdirect composition of a given De~Morgan lattice, $\True$ is the top element and $\Neither$ and $\Both$ are fixpoints of De~Morgan \mbox{negation}. Moreover, if both of the constants $\Neither$ and $\Both$ are present in the signature, then they are Boolean complements in each subdirect factor, hence each subdirect factor is isomorphic to $\DMfour$ and $\Neither$~and~$\Both$ are distinct fixpoints in it. To obtain a subdirect decomposition in the signature which includes these constants, it suffices to apply the isomorphism of $\DMfour$ which exchanges $\Neither$ and $\Both$ to some of the factors.

\subsection{Belnap--Dunn logic}
\label{subsec: bd}

  Finally, we review some facts about Belnap--Dunn logic and its extensions. The proofs of most these facts may be found in the paper~\cite{font97}, which was the first thorough analysis of this logic from the perspective of abstract algebraic logic. The~remaining proofs can be found in the papers~\cite{pietz+rivieccio13,rivieccio12,pynko95b}.
  
  Let us first properly define Belnap--Dunn logic $\BD$, the strong three-valued Kleene logic $\K$, the Logic of Paradox~$\LP$, and Exactly True Logic~$\ETL$:
\begin{align*}
  \BD & \text{ is the logic defined by } \pair{\DMfour}{\{ \True, \Both \}}, \\
  \ETL & \text{ is the logic defined by } \pair{\DMfour}{\{ \True \}}, \\
  \K & \text{ is the logic defined by } \pair{\Kthree}{\{ \True \}}, \\
  \LP & \text{ is the logic defined by } \pair{\Kthree}{\{ \True, \Inter \}},
\end{align*}
  where we take the relational signature to be $\{ \IsTrue \}$. The structures which define $\K$ and $\LP$ are substructures of $\pair{\DMfour}{\{ \True, \Both \}}$, therefore $\K$ and $\LP$ are extensions of $\BD$. The logic $\ETL$ is also an extension of $\BD$: since the isomorphic structures $\pair{\DMfour}{\{ \True, \Both \}}$ and $\pair{\DMfour}{\{ \True, \Neither \}}$ are models of $\BD$, so is $\pair{\DMfour}{\{ \True, \Both \} \cap \{ \True, \Neither \}}$.

  Observe that in a relational signature consisting of a single unary predicate, we may disregard the distinction between terms and formulas, identifying the formula $\IsTrue(t)$ directly with the term $t$. Using this convention, $\K$ is the extension of~$\BD$ by the rule $(x \wedge \dmneg x) \vee y \vdash y$, and $\LP$ is the extension of $\BD$ by the rule (axiom) $\emptyset \vdash x \vee \dmneg x$. Finally, $\ETL$ is the extension of $\BD$ by the rule $x \wedge (\dmneg x \vee y) \vdash y$.

  Each reduced model of $\BD$ is a De~Morgan lattice equipped with a lattice filter. Conversely, each De~Morgan lattice with a lattice filter is a model of~$\BD$, although it need not be a reduced model.

\begin{fact}\label{fact: models of bd}
  Let $\pair{\alg{A}}{T}$ be a non-singleton reduced model of the logic~$\logic{L}$.
\begin{enumerate}
\item[(i)] If $\logic{L} =\BD$, then $\alg{A}$ is a De~Morgan lattice and $T$ is a lattice filter on $\alg{A}$.
\item[(ii)] If $\logic{L} = \ETL$, then $\alg{A}$ is a De~Morgan lattice and $T = \{ \True \}$.
\item[(iii)] If $\logic{L} = \LP$, then $\alg{A}$ is a Kleene lattice and $T$ is a lattice filter on $\alg{A}$.
\item[(iv)] If $\logic{L} = \K$, then $\alg{A}$ is a Kleene lattice and $T = \{ \True \}$.
\end{enumerate}
\end{fact}

\begin{fact}\label{fact: leibniz of bd}
  Let $\pair{\alg{A}}{T}$ be a model of~$\BD$. Then $\pair{a}{b} \in \Leibniz{\alg{A}}{T}$ if and only if
\begin{align*}
  (a \vee c \in T \iff b \vee c \in T) \text{ and } (\dmneg a \vee c \in T \iff \dmneg b \vee c \in T) \text{ for all } c \in \alg{A}.
\end{align*}
\end{fact}

  The following two useful facts now follow from the above description of Leibniz congruences of models of $\BD$.

\begin{fact}\label{fact: prime leibniz reduct}
  Let $T$ be a prime filter on a De~Morgan lattice $\alg{A}$. Then the Leibniz reduct of the structure $\pair{\alg{A}}{T}$ is isomorphic to a substructure of $\pair{\DMfour}{\{ \True, \Both \}}$.
\end{fact}

\begin{proof}
  Fact~\ref{fact: leibniz of bd} can be restated as follows for $T$ prime: $\pair{a}{b} \in \Leibniz{\alg{A}}{T}$ if and only~if
  \begin{align*}
    (a \in T \iff b \in T) \text{ and } (\dmneg a \in T \iff \dmneg b \in T).
  \end{align*}
  But then the kernel of the homomorphism $h_{T}$ (recall Fact~\ref{fact: hf}) is precisely $\Leibniz{\alg{A}}{T}$, and moreover $h_{T}(a) \in \{ \True, \Both \}$ if and only if $a \in T$. The~Leibniz reduct of $\pair{\alg{A}}{T}$ is therefore isomorphic to a substructure of~$\pair{\DMfour}{\{ \True, \Both \}}$.
\end{proof}

\begin{fact}\label{fact: prime leibniz reduct 2}
  Let $T$ and $NF$ be prime filters on a De~Morgan lattice $\alg{A}$ such that $NF = \alg{A} \setminus \dmneg [T]$, where $\dmneg [T] \assign \set{\dmneg a}{a \in T}$. Then $\Leibniz{\alg{A}}{T} = \Leibniz{\alg{A}}{NF}$ and the Leibniz reduct of $\triple{\alg{A}}{T}{NF}$ is isomorphic to a substructure of $\triple{\DMfour}{\{ \True, \Both \}}{\{ \True, \Neither \}}$.
\end{fact}

\begin{proof}
  As in Fact~\ref{fact: leibniz of bd}, $\Leibniz{\alg{A}}{T}$ is the kernel of $h_{T}$, and moreover $h_{T}(a) \in \{ \True, \Both \}$ if and only if $a \in T$. Since $NF = \alg{A} \setminus \dmneg[T]$, we have $\Leibniz{\alg{A}}{NF} = \Leibniz{\alg{A}}{T}$ by Fact~\ref{fact: leibniz of bd}, and moreover $h_{T}(a) \in \{ \True, \Neither \}$ if and only if $a \in NF$. The Leibniz reduct of $\triple{\alg{A}}{T}{NF}$ is therefore isomorphic to a substructure of $\triple{\DMfour}{\{ \True, \Both \}}{\{ \True, \Neither \}}$.
\end{proof}

\section{Completeness without equality}
\label{sec: without equality}

  In this section we axiomatize the logics defined by the following structures:
  \begin{gather*}
    \triple{\DMfour}{\{ \True, \Both \}}{\{ \True \}}, \\
    \triple{\DMfour}{\{ \True, \Both \}}{\{ \True, \Neither \}}, \\
    \triple{\Kthree}{\{ \True, \Neither \}}{\{ \True \}} \iso \triple{\Kthree}{\{ \True, \Both \}}{\{ \True \}}.
  \end{gather*}
  These may naturally be called the four-valued (Belnap--Dunn) logic of truth and exact truth, the four-valued (Belnap--Dunn) logic of truth and non-falsity, and either the three-valued logic of truth and non-falsity or the three-valued logic of truth and exact truth, depending on the interpretation of the middle value.
  
  In the structure $\triple{\DMfour}{\{ \True, \Both \}}{\{ \True \}}$, we view $\{ \True, \Both \}$ as the interpretation of the relational symbol $\IsTrue$ and $\{ \True \}$ as the interpretation of~$\ExTrue$. Likewise, in the structure $\triple{\DMfour}{\{ \True, \Both \}}{\{ \True, \Neither \}}$ we view $\{ \True, \Both \}$ as the interpretation of $\IsTrue$ and $\{ \True, \Neither \}$ as the interpretation of~$\NonFalse$. Finally, in the structure $\triple{\Kthree}{\{ \True, \Neither \}}{\{ \True \}}$ we view $\{ \True, \Neither \}$ and $\{ \True \}$ as interpretations of $\NonFalse$ and $\IsTrue$ respectively, while in the structure $\triple{\Kthree}{\{ \True, \Both \}}{\{ \True \}}$ we view $\{ \True, \Both \}$ and $\{ \True \}$ as interpretations of $\IsTrue$ and $\ExTrue$.

  The completeness proofs all follow the same outline: we axiomatically define a logic~$\logic{L}$ such that the structure in question is a model of~$\logic{L}$. Given a rule which fails in some model of~$\logic{L}$, we then use the closure of the class of all models of~a logic under certain constructions (see Section~\ref{sec: preliminaries}) to show that the rule in fact fails in one of the structures listed above. This will show that a rule holds in the axiomatically defined logic~$\logic{L}$ if and only if it holds in the given structure.

  Each axiomatization will consist of axiomatizations of the individual predicates plus some interaction axioms. When we talk about ``the rules of $\BD$ for $\IsTrue$'' or ``the rules of $\ETL$ for $\ExTrue$'', for example, we mean any axiomatization of the logic of $\pair{\DMfour}{\{ \True, \Both \}}$ where $\{ \True, \Both \}$ is the interpretation of $\IsTrue$ and any axiomatization of the logic of $\pair{\DMfour}{\{ \True \}}$ where $\{ \True \}$ is the interpretation of $\ExTrue$.

\begin{theorem}[Completeness theorem for the logic of $\IsTrue$ and $\ExTrue$]\label{thm: bde}~\\
  The logic of the structure $\triple{\DMfour}{\{ \True, \Both \}}{\{ \True \}}$ in the signature $\{\IsTrue, \ExTrue \}$ is \mbox{axiomatized} by the rules of $\BD$ for $\IsTrue$, the rules of $\ETL$ for $\ExTrue$, and the rules
\begin{align*}
  \ExTrue(x) & \vdash \IsTrue(x), \\
  \ExTrue(x) ~ \& ~ \IsTrue(\dmneg x \vee y) & \vdash \IsTrue(y), \\
  \IsTrue (x) ~ \& ~ \IsTrue (y) ~ \& ~ \ExTrue (\dmneg x \vee y) & \vdash \ExTrue(y).
\end{align*}
  Each reduced model of this logic has the form $\triple{\alg{A}}{T}{\{ \True \}}$ or $\triple{\alg{A}}{T}{\emptyset}$, where $\alg{A}$ is a De~Morgan lattice and $T$ is a lattice filter on $\alg{A}$.
\end{theorem}

\begin{proof}
  Let $\logic{L}$ be the logic axiomatized by the given rules. We first show that each reduced model $\triple{\alg{A}}{T}{E}$ of $\logic{L}$ has the stated form. The structures $\pair{\alg{A}}{T}$ and $\pair{\alg{A}}{E}$ are models of $\BD$, therefore $\alg{A} / \Leibniz{\alg{A}}{T}$ and $\alg{A} / \Leibniz{\alg{A}}{E}$ are De~Morgan lattices by Fact~\ref{fact: models of bd}. It follows that $\alg{A} \iso \alg{A} / \left( \Leibniz{\alg{A}}{T} \cap \Leibniz{\alg{A}}{E} \right)$ is a De~Morgan lattice, by virtue of being a subdirect product of De~Morgan lattices, and $T$ and $E$ are \mbox{lattice} filters on $\alg{A}$. If $T = \emptyset$, then $E = \emptyset$ thanks to the rule $\ExTrue(x) \vdash \IsTrue(x)$. In that case the reduced structure $\triple{\alg{A}}{T}{E} = \triple{\alg{A}}{\emptyset}{\emptyset}$ is isomorphic to the substructure of $\triple{\DMfour}{\{ \True, \Both \}}{\{ \True \}}$ with the universe~$\{ \Neither \}$.

  It remains to prove that $E$ is a singleton if non-empty. Consider thus $a, b \in E$. Then $\pair{a}{b} \in \Leibniz{\alg{A}}{T}$: $a \vee c \in T$ and $b \vee c \in T$ for each $c \in \alg{A}$ by virtue of the rule $\ExTrue(x) \vdash \IsTrue(x)$, and $\dmneg a \vee c \in T \iff \dmneg b \vee c \in T$ for each $c \in \alg{A}$ thanks to the rule $\ExTrue(x) ~ \& ~ \IsTrue(\dmneg x \vee y) \vdash \IsTrue(y)$. Likewise, $\pair{a}{b} \in \Leibniz{\alg{A}}{E}$: $a \vee c \in E$ and $b \vee c \in E$ for each $c \in \alg{A}$, and $\dmneg a \vee c \in E \iff \dmneg b \vee c \in E$ thanks to the rule $\ExTrue(x) ~ \& ~ \ExTrue(\dmneg x \vee y) \vdash \ExTrue(y)$. Thus $a = b$ and $E = \{ \True \}$.

  The structure $\triple{\DMfour}{\{ \True, \Both \}}{\{ \True \}}$ is a model of $\logic{L}$. Conversely, suppose that a rule fails in some model of $\logic{L}$. Then it fails in a non-trivial reduced model $\triple{\alg{A}}{T}{E}$, where $\alg{A}$ is a De~Morgan lattice, $T$ is a lattice filter on $\alg{A}$, and $E = \{ \True \}$ or $E = \emptyset$. In the rest of the proof, we show that the rule also fails in $\triple{\DMfour}{\{ \True, \Both \}}{\{ \True \}}$.

  Consider $a \notin T$ (if some such $a$ exists). The ideal generated by adding $a$ to $\dmneg [E]$ is disjoint from the filter~$T$: if $\dmneg c \vee a \in T$ for some $c \in E$, then $a \in T$ thanks to the rule $\ExTrue(x) ~ \& ~ \IsTrue(\dmneg x \vee y) \vdash \IsTrue(y)$. This ideal extends to a prime ideal $I_{a}$ disjoint from $T$ by Lemma~\ref{lemma: filter-ideal separation}.
  
   Similarly, consider $b \in T \setminus E$ (if some such $b$ exists). The ideal generated by adding $b$ to $\dmneg [T]$ is disjoint from the filter $E$: if $\dmneg c \vee b \in E$ for some $c \in T$, then $b \in E$ thanks to the rule $\IsTrue(x) ~ \& ~ \IsTrue(y) ~ \& ~ \ExTrue(\dmneg x \vee y) \vdash \ExTrue(y)$. This ideal extends to a prime ideal $J_{b}$ disjoint from $E$ by Lemma~\ref{lemma: filter-ideal separation}.
  
  Let us now define the following prime filters for $a \notin T$ and $b \in T \setminus E$:
  \begin{align*}
    T_{a} & \assign \alg{A} \setminus I_{a} \supseteq T, & U_{b} & \assign \dmneg [J_{b}] \supseteq T, \\
    M_{a} & \assign \dmneg [I_{a}] \supseteq E, & N_{b} & \assign \alg{A} \setminus J_{b} \supseteq E.
  \end{align*}
  In particular, $T_{a} \cap \dmneg [M_{a}] = \emptyset$ and $U_{b} \cap \dmneg [N_{b}] = \emptyset$. Moreover, $a \notin T_{a}$ and $b \notin N_{b}$. It~follows that
  \begin{align*}
    T & = \left( \bigcap_{a \in \alg{A} \setminus T} T_{a} \right) \cap \left( \bigcap_{b \in T \setminus E} U_{b} \right), \\
    E & = \left( \bigcap_{a \in \alg{A} \setminus T} (T_{a} \cap M_{a}) \right) \cap \left( \bigcap_{b \in T \setminus E} (U_{b} \cap N_{b}) \right).
  \end{align*}
  There is either some $a \in \alg{A} \setminus T$ or some $b \in T \setminus E$ because the structure $\triple{\alg{A}}{T}{E}$ is a non-trivial. If a rule fails in the structure $\triple{\alg{A}}{T}{E}$, it therefore fails in one of the structures $\triple{\alg{A}}{T_{a}}{T_{a} \cap M_{a}}$ or $\triple{\alg{A}}{U_{b}}{U_{b} \cap N_{b}}$ (recall Fact~\ref{fact: intersection of filters}). Let $T' \assign T_{a}$ and $NF' \assign M_{a}$ in the former case, and $T' \assign U_{b}$ and $NF' \assign N_{b}$ in the latter case. Then the rule in question fails in the structure $\triple{\alg{A}}{T'}{T' \cap NF'}$, where $T' \supseteq T$ and $NF' \supseteq E$ are prime filters on $\alg{A}$ with $NF' = \alg{A} \setminus \dmneg [T']$.
  
  Now consider the Leibniz congruence $\theta \assign \Leibniz{\alg{A}}{T'} \cap \Leibniz{\alg{A}}{NF'}$ of the structure $\triple{\alg{A}}{T'}{NF'}$. Because
  \begin{align*}
    \theta = \Leibniz{\alg{A}}{T'} \cap (\Leibniz{\alg{A}}{T'} \cap \Leibniz{\alg{A}}{NF'}) \subseteq \Leibniz{\alg{A}}{T'} \cap \Leibniz{\alg{A}}{(T' \cap NF')},
  \end{align*}
  the structure $\triple{\alg{A} / \theta}{T' / \theta}{(T' \cap NF') / \theta}$ validates the same rules as $\triple{\alg{A}}{T'}{T' \cap NF'}$ by Fact~\ref{fact: leibniz}. But by Fact~\ref{fact: prime leibniz reduct 2} the structure $\triple{\alg{A} / \theta}{T' / \theta}{NF' / \theta}$ is iso\-morphic to a substructure of the structure $\triple{\DMfour}{\{ \True, \Both \}}{\{ \True, \Neither \}}$, hence $\triple{\alg{A} / \theta}{T' / \theta}{(T' \cap NF') / \theta}$ is isomorphic to a substructure of $\triple{\DMfour}{\{ \True, \Both \}}{\{ \True \}}$.

  (Note that we did \emph{not} take the Leibniz reduct of the structure $\triple{\alg{A}}{T'}{T' \cap NF'}$. Instead, we factored it by the Leibniz congruence of $\triple{\alg{A}}{T'}{NF'}$ because we wanted to factor by a congruence compatible with~$NF'$.)
\end{proof}

  The above axiomatization is in fact slightly redundant. Recall that if we identify terms and formulas as in Subsection~\ref{subsec: bd}, then $\ETL$ is the extension of $\BD$ by the rule $x \wedge (\dmneg x \vee y) \vdash y$. But the rule $\ExTrue(x \wedge (\dmneg x \vee y)) \vdash \ExTrue(y)$ is derivable from the three inter\-action axioms between $\IsTrue$ and $\ExTrue$. When stating the axioms for the predicate $\ExTrue$ in Theorem~\ref{thm: bde}, it therefore suffices to take an axiomatization of $\BD$ and replace the predicate $\IsTrue$ by $\ExTrue$ throughout.

\begin{theorem}\label{thm: bde exp}
  In order to axiomatize the expansion of the above logic by some of the constants $\True$, $\Neither$, $\Both$, it suffices to add the appropriate subset of the following rules:
  \begin{align*}
    \emptyset & \vdash \ExTrue(\True), & \IsTrue(\Neither \vee x) & \vdash \IsTrue(x), & \IsTrue(x) & \vdash \ExTrue(\Neither \vee x), \\
    \emptyset & \vdash \IsTrue(\Both \wedge \dmneg \Both), & \IsTrue(\dmneg \Neither \vee x) & \vdash \IsTrue(x), & \IsTrue(x) & \vdash \ExTrue(\dmneg \Neither \vee x).
  \end{align*}
\end{theorem}

\begin{proof}
  It will suffice to extend or modify the proof of Theorem~\ref{thm: bde} in order to ensure that $\True \in T'$ and $\dmneg \True \notin T'$, $\Both \in T'$ and $\dmneg \Both \in T'$, and $\Neither \notin T'$ and $\dmneg \Neither \notin T'$. This is because the structure $\triple{\alg{A} / \theta}{T' / \theta}{NF' / \theta}$ is then isomorphic to a substructure of $\triple{\DMfour}{\{ \True, \Both \}}{\{ \True, \Neither \}}$, and therefore $\triple{\alg{A} / \theta}{T' / \theta}{(T' \cap NF') / \theta}$ is isomorphic to a substructure of $\triple{\DMfour}{\{ \True, \Both \}}{\{ \True \}}$ even if the signature is expanded by the constants.

  For the constants $\True$ and $\Both$ this is easy. The rules directly ensure that $\True \in T'$ and $\True \in NF'$, hence $\dmneg \True \notin T'$, and also $\Both, \dmneg \Both \in T'$. For the constant $\Neither$, we modify the proof of Theorem~\ref{thm: bde} slightly.

  Instead of considering the ideal generated by adding $a \notin T$ to $\dmneg [E]$, we consider the ideal generated by adding $\Neither \vee \dmneg \Neither \vee a$ to $\dmneg [E]$. This is an ideal disjoint from~$T$ thanks to the rules $\IsTrue(\Neither \vee x) \vdash \IsTrue(x)$ and $\IsTrue(\dmneg \Neither \vee x) \vdash \IsTrue(x)$. It therefore extends to a prime ideal $I_{a}$ disjoint from $T$.
  
   Likewise, instead of proving that the ideal generated by adding $b \in T \setminus E$ to $\dmneg [T]$ is disjoint from $E$, we prove that it is disjoint from the filter generated by adding $\Neither \wedge \dmneg \Neither$ to $E = \{ \True \}$ (which is just the filter generated by $\Neither \wedge \dmneg \Neither$). If it were not, then $\Neither \wedge \dmneg \Neither \leq \dmneg c \vee b$ for some $c \in T$, but then $\dmneg c \vee b = (\Neither \wedge \dmneg \Neither) \vee \dmneg c \vee b \in E$ because $\Neither \vee \dmneg c \vee b \in E$ and $\dmneg \Neither \vee \dmneg c \vee b \in E$ thanks to the rules $\IsTrue(x) \vdash \ExTrue(\Neither \vee x)$ and $\IsTrue(x) \vdash \ExTrue(\dmneg \Neither \vee x)$. The ideal generated by adding $b \in T \setminus E$ therefore extends to a prime filter $J_{b}$ disjoint from the filter generated by $\Neither \wedge \dmneg \Neither$.

 The above modifications then ensure that $\Neither \notin T'$ and $\dmneg \Neither \notin T'$.
\end{proof}

  The completeness theorem for the four-valued logic of truth and non-falsity is proved in a very similar fashion.

\begin{theorem}[Completeness theorem for the logic of $\IsTrue$ and $\NonFalse$]\label{thm: bdnf}~\\
  The logic of the structure $\triple{\DMfour}{\{ \True, \Both \}}{\{ \True, \Neither \}}$ in the signature $\{\IsTrue, \NonFalse \}$ is axiomatized by the rules of $\BD$ for $\IsTrue$, the rules of $\BD$ for $\NonFalse$, and the rules
\begin{align*}
  \IsTrue(x) ~ \& ~ \NonFalse(\dmneg x \vee y) & \vdash \NonFalse(y), \\
  \NonFalse(x) ~ \& ~ \IsTrue(\dmneg x \vee y) & \vdash \IsTrue(y).
\end{align*}
  Each reduced model of this logic has the form $\triple{\alg{A}}{T}{NF}$, where $\alg{A}$ is a De~Morgan lattice with lattice filters $T$ and $NF$, and $T \cap NF = \{ \True \}$ or $T \cap NF = \emptyset$.
\end{theorem}

\begin{proof}
  Let $\logic{L}$ be the logic axiomatized by these rules. We show that each reduced model $\triple{\alg{A}}{T}{NF}$ of $\logic{L}$ has the stated form. The structures $\pair{\alg{A}}{T}$ and $\pair{\alg{A}}{NF}$ are models of $\BD$, hence $\alg{A} / \Leibniz{\alg{A}}{T}$ and $\alg{A} / \Leibniz{\alg{A}}{NF}$ are De~Morgan lattices by Fact~\ref{fact: models of bd}. Thus $\alg{A} \iso \alg{A} / \left( \Leibniz{\alg{A}}{T} \cap \Leibniz{\alg{A}}{NF} \right)$ is a De~Morgan lattice and $T$ and $NF$ are lattice filters.

  To prove that $T \cap NF$ is a singleton if non-empty, consider $a, b \in T \cap NF$. Then $\pair{a}{b} \in \Leibniz{\alg{A}}{T}$: $a \vee c \in T$ and $b \vee c \in T$ for each $c \in \alg{A}$, and $\dmneg a \vee c \in T \iff \dmneg b \vee c \in T$ for each $c \in \alg{A}$ thanks to the rule $\NonFalse(x) ~ \& ~ \IsTrue(\dmneg x \vee y) \vdash \IsTrue(y)$. Similarly, $\pair{a}{b} \in \Leibniz{\alg{A}}{NF}$ thanks to the rule $\IsTrue(x) ~ \& ~ \NonFalse(\dmneg x \vee y) \vdash \NonFalse(y)$.

  The structure $\triple{\DMfour}{\{ \True, \Both \}}{\{ \True, \Neither \}}$ is a model of $\logic{L}$. Conversely, suppose that a rule fails in some model of $\logic{L}$. Then it fails in a non-trivial reduced model $\triple{\alg{A}}{T}{NF}$. In~the rest of the proof, we show that the rule also fails in the structure $\triple{\DMfour}{\{ \True, \Both \}}{\{ \True, \Neither \}}$.

  Consider $a \notin T$ (if some such $a$ exists). The ideal generated by adding $a$ to $\dmneg [NF]$ is disjoint from the filter~$T$: if $\dmneg c \vee a \in T$ for some $c \in NF$, then $a \in T$ thanks to the rule $\NonFalse(x) ~ \& ~ \IsTrue(\dmneg x \vee y) \vdash \IsTrue(y)$. This ideal extends to a prime ideal $I_{a}$ disjoint from $T$ by Lemma~\ref{lemma: filter-ideal separation}.
  
   Similarly, consider $b \notin NF$ (if some such $b$ exists). The ideal generated by adding $b$ to $\dmneg [T]$ is disjoint from $NF$ by the rule $\IsTrue(x), \NonFalse(\dmneg x \vee y) \vdash \NonFalse(y)$. This ideal extends to a prime ideal $J_{b}$ disjoint from $NF$ by Lemma~\ref{lemma: filter-ideal separation}.

  Let us now define the following prime filters for $a \notin T$ and $b \notin NF$:
  \begin{align*}
    T_{a} & \assign \alg{A} \setminus I_{a} \supseteq T, & U_{b} & \assign \alg{A} \setminus J_{b} \supseteq T, \\
    M_{a} & \assign \dmneg [I_{a}] \supseteq NF, & N_{b} & \assign \dmneg [J_{b}] \supseteq NF.
  \end{align*}
  In particular, $T_{a} \cap \dmneg [M_{a}] = \emptyset$ and $U_{b} \cap \dmneg [N_{b}] = \emptyset$. Moreover, $a \notin T_{a}$ and $b \notin N_{b}$. It~follows that
  \begin{align*}
    T & = \left( \bigcap_{a \in \alg{A} \setminus T} T_{a} \right) \cap \left( \bigcap_{b \in \alg{A} \setminus NF} U_{b} \right), \\
    NF & = \left( \bigcap_{a \in \alg{A} \setminus T} M_{a} \right) \cap \left( \bigcap_{b \in \alg{A} \setminus NF} N_{b} \right).
  \end{align*}
  There is either some $a \in \alg{A} \setminus T$ or some $b \in \alg{A} \setminus NF$ because the structure $\triple{\alg{A}}{T}{NF}$ is non-trivial. If a rule fails in the structure $\triple{\alg{A}}{T}{NF}$, it therefore also fails in one of the structures $\triple{\alg{A}}{T_{a}}{M_{a}}$ or $\triple{\alg{A}}{U_{b}}{N_{b}}$ (recall Fact~\ref{fact: intersection of filters}). Let $T' \assign T_{a}$ and $NF' \assign M_{a}$ in the former case, and $T' \assign U_{b}$ and $NF' \assign N_{b}$ in the latter case. Then the rule in question fails in the structure $\triple{\alg{A}}{T'}{NF'}$, where $T' \supseteq T$ and $NF' \supseteq NF$ are prime filters on $\alg{A}$ with $NF' = \alg{A} \setminus \dmneg [T']$. By Fact~\ref{fact: prime leibniz reduct 2}, the Leibniz reduct of $\triple{\alg{A}}{T'}{NF'}$ is now isomorphic to a substructure of $\triple{\DMfour}{\{ \True, \Both \}}{\{ \True, \Neither \}}$. Thus each rule which fails in $\triple{\alg{A}}{T'}{NF'}$ also fails in $\triple{\DMfour}{\{ \True, \Both \}}{\{ \True, \Neither \}}$.
\end{proof}

\begin{theorem}\label{thm: bdnf exp}
  In order to axiomatize the expansion of the above logic by some of the constants $\True$, $\Neither$, $\Both$, it suffices to add the appropriate subset of the following rules:
  \begin{align*}
    \emptyset & \vdash \IsTrue(\True), & \emptyset & \vdash \IsTrue(\Both), & \emptyset & \vdash \NonFalse(\Neither), \\
    \emptyset & \vdash \NonFalse(\True), & \emptyset & \vdash \IsTrue(\dmneg \Both), & \emptyset & \vdash \NonFalse(\dmneg \Neither).
  \end{align*}
\end{theorem}

\begin{proof}
  The argument is the same as in Theorem~\ref{thm: bde exp}: these rules suffice to ensure that the constants belong or do not belong to $T'$ and $NF'$ as appropriate.
\end{proof}

\begin{theorem}[Completeness theorem for the Kleene logic of $\IsTrue$ and $\ExTrue$]\label{thm: ke}~\\
  The logic of the structure $\triple{\Kthree}{\{ \True, \Both \}}{\{ \True \}}$ in the signature $\{ \IsTrue, \ExTrue \}$ is axiomatized by the rules of $\BD$ for $\IsTrue$, the rules of $\ETL$ for $\ExTrue$, and the rules
\begin{align*}
  \emptyset & \vdash \IsTrue(x \vee \dmneg x), & \ExTrue(x) ~ \& ~ \IsTrue(\dmneg x \vee y) & \vdash \IsTrue(y), \\
  \ExTrue(x) & \vdash \IsTrue(x), & \IsTrue (x) ~ \& ~ \ExTrue (\dmneg x \vee y) & \vdash \ExTrue(y).
\end{align*}
  Each reduced model of this logic has the form $\triple{\alg{A}}{T}{\{ \True \}}$ or $\triple{\alg{A}}{T}{\emptyset}$, where $\alg{A}$ is a Kleene lattice and $T$ is a lattice filter on $\alg{A}$.
\end{theorem}

\begin{proof}
  The proof of this theorem is entirely analogous to the proof of Theorem~\ref{thm: bde}. Having already gone through two variants of this proof in detail, we only show how the proof of Theorem~\ref{thm: bde} has to be modified.

  If $\logic{L}$ is the logic axiomatized by the given rules and $\triple{\alg{A}}{T}{E}$ is a reduced model of $\logic{L}$, then the rule $\emptyset \vdash \IsTrue(x \vee \dmneg x)$ ensures that the structure $\pair{\alg{A}}{T}$ is a model of~$\LP$, while the other rules ensure that $\pair{\alg{A}}{E}$ is a model of $\K$: if $(a \wedge \dmneg a) \vee b \in E$, then $b \in E$ because $a \wedge \dmneg a = \dmneg (a \vee \dmneg a)$ and $a \vee \dmneg a \in T$. Therefore $\alg{A}$ is a Kleene lattice by Fact~\ref{fact: models of bd}, and $T$ is a lattice filter and $E = \{ \True \}$ or $E = \emptyset$ by Theorem~\ref{thm: bde}.

  The rest of the proof proceeds as in Theorems~\ref{thm: bde} and \ref{thm: bdnf}. The only difference is that thanks to the rule $\emptyset \vdash \IsTrue(x \vee \dmneg x)$ we can restrict to the substructure of $\triple{\DMfour}{\{ \True, \Both \}}{\{ \True \}}$ with the universe $\{ \False, \Both, \Neither \}$.
\end{proof}

\begin{theorem}
  In order to axiomatize the expansion of the above logic by some of the constants $\True$ and $\Both$, it suffices to add the appropriate subset of the following rules:
  \begin{align*}
    \emptyset & \vdash \ExTrue(\True), &
    \emptyset & \vdash \IsTrue(\Both), &
    \emptyset & \vdash \IsTrue(\dmneg \Both).
  \end{align*}
\end{theorem}

\begin{proof}
  The argument is again the same as in Theorem~\ref{thm: bde exp}.
\end{proof}

  The combination of all three predicates $\{ \IsTrue, \NonFalse, \ExTrue \}$ need not be considered explicitly, since the predicate $\ExTrue$ is definable in terms of the other \mbox{predicates}: $\IsTrue(x) ~ \& ~ \NonFalse \vdash \ExTrue(x)$ and conversely $\ExTrue(x) \vdash \IsTrue(x)$ and $\ExTrue(x) \vdash \NonFalse(x)$.

\section{Completeness with equality}
\label{sec: with equality}

  We now add the predicate $\approx$, interpreted by the equality relation on the algebras $\DMfour$ and $\Kthree$, to the logics introduced in the previous section and to $\BD$ and $\ETL$.
   
  Let us first observe that equality in any structure is a compatible congruence (see Subsection~\ref{subsec: leibniz}), which means precisely that it satisfies the following rules stating that it is an equivalence relation:
  \begin{align*}
    \emptyset & \vdash p \approx p, & p \approx q & \vdash q \approx p, & p \approx q ~ \& ~ q \approx r & \vdash p \approx r,
  \end{align*}
  as well as the following rules for each $n$-ary function symbol $f$ and each $n$-ary \mbox{relation} symbol $\Rrel$, which state that it is a congruence with respect to $f$ and that it is compatible with $\Rrel$:
  \begin{align*}
    p_{1} \approx q_{1} ~ \& ~ \dots ~ \& ~ p_{n} \approx q_{n} & \vdash f(p_{1}, \dots, p_{n}) \approx f(q_{1}, \dots, q_{n}), \\
    \Rrel(p_{1}, \dots, p_{n}), p_{1} \approx q_{1} ~ \& ~ \dots ~ \& ~ p_{n} \approx q_{n} & \vdash \Rrel(q_{1}, \dots, q_{n}).
  \end{align*}
  Conversely, any binary relation on a \emph{reduced} structure which satisfies the above rules is in fact the equality relation. Let us record this fact here.

\begin{fact} \label{fact: equality relation}
  A compatible congruence on a reduced structure is the equality relation.
\end{fact}

  In universal algebra, the above rules are taken for granted because the relational symbol $\approx$ is assumed to be interpreted by the equality relation. To axiomatize the quasi-equational theory of $\DMfour$ in the sense of universal algebra, the equalities $\dmneg \dmneg x \approx x$ and either $\dmneg (x \vee y) \approx \dmneg x \wedge \dmneg y$ or $\dmneg (x \wedge y) \approx \dmneg x \vee \dmneg y$ therefore suffice.

  In the current setting, on the other hand, we treat the equality symbol on an equal footing with other relational symbols, therefore the above rules must be imposed explicitly. An axiomatization of the logic of the structure $\pair{\DMfour}{=}$ in our sense therefore consists of the above compatibility rules plus the rules $\emptyset \vdash \dmneg \dmneg x \approx x$ and either $\emptyset \vdash \dmneg (x \vee y) \approx \dmneg x \wedge \dmneg y$ or $\emptyset \vdash \dmneg (x \wedge y) \approx \dmneg x \vee \dmneg y$.

\begin{theorem}[Completeness theorem for the logic of $\IsTrue$ and $\approx$]\label{thm: bd eq}~\\
  The logic of the structure $\triple{\DMfour}{\{ \True, \Both \}}{=}$ in the signature $\{ \IsTrue, \approx \}$ is axiomatized by extending an axiomatization of De~Morgan lattices for $\approx$ by the rules
\begin{align*}
  \IsTrue(x) ~ \& ~ \IsTrue(y) & \vdash \IsTrue(x \wedge y), \\
  \IsTrue(x) ~ \& ~ \IsTrue(y) & \vdash \dmneg x \vee y \approx y,
\end{align*}
\vskip -20pt
\begin{align*}
  \IsTrue(z) ~ \& ~ x \wedge z \approx y \wedge z ~ \& ~ \dmneg y \wedge z \approx \dmneg x \wedge z & \vdash x \approx y.
\end{align*}
  Each reduced model of this logic has the form $\triple{\alg{A}}{T}{=}$, where $\alg{A}$ is a De~Morgan lattice and $T$ is a lattice filter on $\alg{A}$.
\end{theorem}

\begin{proof}
  Let $\logic{L}$ be the logic axiomatized by the given rules and $\triple{\alg{A}}{T}{\approx}$ be a reduced model of $\logic{L}$. Then ${\approx}$ is the equality relation on $\alg{A}$ by Fact~\ref{fact: equality relation}, therefore $\alg{A}$ is a De~Morgan lattice and $T$ is a lattice filter on $\alg{A}$.

  The structure $\triple{\DMfour}{\{ \True, \Both \}}{=}$ is a model of $\logic{L}$. Conversely, we prove that if a rule fails in a non-trivial reduced model $\triple{\alg{A}}{T}{\approx}$ of $\logic{L}$, then it fails in $\triple{\DMfour}{\{ \True, \Both \}}{=}$.

  We claim that for each $a < b$ in $\alg{A}$ there is some filter $U \supseteq T$ such that either $b \in U$ and $a \notin U$ or $\dmneg a \in U$ and $\dmneg b \notin U$. (We may assume without loss of generality that $a \in T \iff b \in T$ and $\dmneg a \in T \iff \dmneg b \in T$, otherwise the claim holds trivially for $U \assign T$.) The following case analysis proves this claim:
  
  (i) If $a, \dmneg a, b, \dmneg b \in T$, then $a = b$ because $\dmneg (\dmneg a) \leq b$ and $\dmneg (\dmneg b) \leq a$ thanks to the rule $\IsTrue(x) ~ \& ~ \IsTrue(y) \vdash \dmneg x \vee y \approx y$.

  (ii) If $a, b \in T$ and $\dmneg a, \dmneg b \notin T$, then there is some filter $U \supseteq T$ such that $\dmneg a \in U$ and $\dmneg b \notin U$. Otherwise the filter generated by $T$ and $\dmneg a$ contains $\dmneg b$, in which case there is some $f \in T$ such that $f \wedge \dmneg a \leq \dmneg b$, so $b \leq a \vee \dmneg f$. But $a \vee \dmneg f = a$ thanks to the rule $\IsTrue(x) ~ \& ~ \IsTrue(y) \vdash \dmneg x \vee y \approx y$, so $b \leq a$, contradicting $a < b$.

  (iii) If $a, b \notin T$ and $\dmneg a, \dmneg b \in T$, then by (ii) with $a, b$ and $\dmneg a, \dmneg b$ exchanged, there is some filter $U \supseteq T$ such that $b \in U$ and $a \notin U$.

  (iv) If $a, \dmneg a, b, \dmneg b \notin T$, then there is some filter $U \supseteq T$ such that either $b \in U$ and $a \notin U$ or $\dmneg a \in U$ and $\dmneg b \notin U$. Otherwise there is some $f \in T$ such that $f \wedge b \leq a$ and $f \wedge \dmneg a \leq \dmneg b$, hence $f \wedge a \wedge b = f \wedge b$ and $f \wedge \dmneg (a \wedge b) = f \wedge \dmneg b$. But then $a \wedge b = b$ thanks to the rule $\IsTrue(z) ~ \& ~ x \wedge z \approx y \wedge z ~ \& ~ \dmneg y \wedge z \approx \dmneg x \wedge z \vdash x \approx y$.
  
  Let $\set{U_{i}}{i \in I}$ be the set of all prime filters extending $T$. The above claim implies by Lemma~\ref{lemma: filter-ideal separation} that for each $a < b$ there is a prime filter $U_{i} \supseteq T$ such that either $b \in U_{i}$ and $a \notin U_{i}$ or $\dmneg a \in U_{i}$ and $\dmneg b \notin U_{i}$. In particular, $\pair{a}{b} \notin \Leibniz{\alg{A}} U_{i}$. Thus $a = b$ in $\alg{A}$ if and only if $\pair{a}{b} \in \Leibniz{\alg{A}} U_{i}$ for each prime filter $U_{i} \supseteq T$. In other words,  
  \begin{align*}
    \triple{\alg{A}}{T}{=} & = \biggl< \alg{A}, \bigcap_{i \in I} U_{i}, \bigcap_{i \in I} (\Leibniz{\alg{A}}{U_{i})} \biggr>.
  \end{align*}
  But then by Fact~\ref{fact: intersection of filters} a rule which fails in $\triple{\alg{A}}{T}{=}$ must also fail in one of the structures $\triple{\alg{A}}{U_{i}}{\Leibniz{\alg{A}} U_{i}}$. Finally, by Fact~\ref{fact: prime leibniz reduct} the Leibniz reduct of $\triple{\alg{A}}{U_{i}}{\Leibniz{\alg{A}} U_{i}}$, namely the structure $\triple{\alg{A} / \Leibniz{\alg{A}} U_{i}}{U_{i} / \Leibniz{\alg{A}}{U_{i}}}{=}$, is isomorphic to a sub\-structure of $\triple{\DMfour}{\{ \True, \Both \}}{=}$, therefore the rule also fails in $\triple{\DMfour}{\{ \True, \Both \}}{=}$.
\end{proof}

\begin{theorem}\label{thm: bd eq exp}
  In order to axiomatize the expansion of the above logic by some of the constants $\True$, $\Neither$, $\Both$, it suffices to add the appropriate subset of the following rules:
\begin{align*}
  \emptyset & \vdash \True \approx \True \vee x, & \emptyset & \vdash \Neither \approx \dmneg \Neither, & \emptyset & \vdash \Both \vee \Neither \approx \True,
\end{align*}
\begin{align*}
  \emptyset & \vdash \IsTrue(\True), & \emptyset & \vdash \IsTrue(\Both), & \IsTrue(\Neither \vee x) & \vdash \IsTrue(x), \\
  \IsTrue(\dmneg \True) & \vdash x \approx y, & \emptyset & \vdash \IsTrue(\dmneg \Both), & \IsTrue(x) & \vdash \Neither \vee x \vee y \approx \Neither \vee x.
\end{align*}
\end{theorem}

\begin{proof}
  In order to prove this theorem, it suffices to extend or modify the previous proof (with $\logic{L}$ now being the logic extended by the above axioms) and prove that we can choose $U$ so that $\True \in U$, $\dmneg \True \notin U$, and $\Neither, \dmneg \Neither \notin U$, and $\Both, \dmneg \Both \in U$.

  The proofs for the constants~$\Both$ and $\True$ merely extend the previous proof. The two axioms for $\Both$ ensure that $\Both \in T \subseteq U$ and $\dmneg \Both \in T \subseteq U$. (Note that the equality $\Both \approx \dmneg \Both$ is derivable from the axioms for $\Both$ given the other rules.) The two axioms for $\True$ ensure that $\True$ is the top element of $\alg{A}$ and that $\True \in T \subseteq U$. The rule $\IsTrue(\dmneg \True) \vdash x \approx y$ ensures that in each reduced model $\triple{\alg{A}}{T}{\approx}$ of $\logic{L}$ either $\dmneg \True \notin T$ as desired or $\alg{A}$ is a singleton. But in the latter case the structure $\triple{\alg{A}}{T}{\approx}$ is trivial.

  It remains to prove the claim for $\Neither$. Since $\Neither = \dmneg \Neither$ in each reduced model $\triple{\alg{A}}{T}{\approx}$ of $\logic{L}$, it suffices to prove that we can choose $U$ so that $\Neither \notin U$. Observe that the axioms for $\Neither$ imply that $\Neither \vee f = \True$ whenever $f \in T$, and moreover
\begin{align*}
  f \wedge a \leq b \text{ and } f \wedge \dmneg b \leq \dmneg a \text{ for some } f \in T \implies a \leq b.
\end{align*}
  Recall also that in a distributive lattice
\begin{align*}
  f \wedge a \leq f \wedge b \text{ and } f \vee a \leq f \vee b \text{ for some } f \implies a \leq b.
\end{align*}

  We claim that if $\Neither \vee a < \Neither \vee b$ in $\alg{A}$, then there is some filter $U \supseteq T$ such that either $\Neither \vee b \in U$ and $\Neither \vee a \notin U$ or $\Neither \vee \dmneg a \in U$ and $\Neither \vee \dmneg b \notin U$. (We may again assume without loss of generality that $\Neither \vee a \in T \iff \Neither \vee b \in T$ and $\Neither \vee \dmneg a \in T \iff \Neither \vee \dmneg b \in T$.) The following case analysis proves this claim:

  (i) If $\Neither \vee a, \Neither \vee \dmneg a, \Neither \vee b, \Neither \vee \dmneg b \in T$, then $a, \dmneg a, b, \dmneg b \in T$ thanks to the rule $\IsTrue(\Neither \vee x) \vdash \IsTrue(x)$. Applying the rule $\IsTrue(x) ~ \& ~ \IsTrue(y) \vdash \dmneg x \vee y \approx y$ then yields that $a = b$.

  (ii) If $\Neither \vee a, \Neither \vee b \in T$ and $\Neither \vee \dmneg a, \Neither \vee \dmneg b \notin T$, then there is a filter $U \supseteq T$ such that $\Neither \vee \dmneg a \in U$ and $\Neither \vee \dmneg b \notin U$. Otherwise the filter generated by $T$ and $\Neither \vee \dmneg a$ contains $\Neither \vee \dmneg b$, in which case there is $f \in T$ such that $f \wedge (\Neither \vee \dmneg a) \leq \Neither \vee \dmneg b$, so $\Neither \wedge b \leq \dmneg f \vee (\Neither \wedge a) \leq \dmneg f \vee a$. But the rule $\IsTrue(x) ~ \& ~ \IsTrue(y) \vdash \dmneg x \vee y \approx y$ ensures that $\dmneg f \vee a = a$, so $\Neither \wedge b \leq \Neither \wedge a$. On the other hand, $\Neither \vee b = \True = \Neither \vee a$ because $a, b \in T$ by the rule $\IsTrue(\Neither \vee x) \vdash \IsTrue(x)$. Therefore $b = a$.

  (iii) If $\Neither \vee a, \Neither \vee b \notin T$ and $\Neither \vee \dmneg a, \Neither \vee \dmneg b \in T$, then by (ii) with $\Neither \vee a, \Neither \vee b$ and $\Neither \vee \dmneg a, \Neither \vee \dmneg b$ exchanged there is some filter $U \supseteq T$ such that $\Neither \vee b \in U$ and $\Neither \vee a \notin U$.

  (iv) If $a, \dmneg a, b, \dmneg b \notin T$, then there is some filter $U \supseteq T$ such that either $\Neither \vee b \in U$ and $\Neither \vee a \notin U$ or $\Neither \vee \dmneg a \in U$ and $\Neither \vee \dmneg b \notin U$. Otherwise there is some $f \in T$ such that $f \wedge (\Neither \vee b) \leq \Neither \vee a$ and $f \wedge \dmneg (\Neither \vee a) = f \wedge (\Neither \vee \dmneg a) \leq \Neither \vee \dmneg b = \dmneg (\Neither \vee b)$. But then $\Neither \vee b \leq \Neither \vee a$ by one of the axioms, hence $\Neither \vee b = \Neither \vee a$.

  Finally, note that if $a < b$, then either $\Neither \vee a < \Neither \vee b$ or $\Neither \wedge a < \Neither \wedge b$. In other words, either $\Neither \vee a < \Neither \vee b$ or $\Neither \vee \dmneg b < \Neither \vee \dmneg a$. For each $a < b$ we can therefore find a filter $U \supseteq T$ such that either $\Neither \vee b \in U$ and $\Neither \vee a \notin U$ or $\Neither \vee \dmneg a \in U$ and $\Neither \vee \dmneg b \notin U$. By Lemma~\ref{lemma: filter-ideal separation} there is again a prime filter $U_{i} \supseteq T$ with this property. But for prime filters the above property means that either $b \in U_{i}$ and $a \notin U_{i}$ or $\dmneg a \in U_{i}$ and $\dmneg b \notin U_{i}$. The~rest of the proof is the same as in Theorem~\ref{thm: bd eq}.
\end{proof}

\begin{theorem}[Completeness theorem for the logic of $\ExTrue$ and $\approx$]\label{thm: etl eq}~\\
  The logic of the structure $\triple{\DMfour}{\{ \True \}}{=}$ in the signature $\{ \ExTrue, \approx \}$ is axiomatized by extending an axiomatization of De~Morgan lattices for $\approx$ by the rule
\begin{align*}
  \ExTrue(x) & \vdash x \vee y \approx x.
\end{align*}
  Each reduced model of this logic has either the form $\triple{\alg{A}}{\{ \True \}}{=}$ or the form $\triple{\alg{A}}{\emptyset}{=}$, where $\alg{A}$ is a De~Morgan lattice.
\end{theorem}

\begin{proof}
  The structure $\triple{\DMfour}{\{ \True \}}{=}$ is a model of these rules. Conversely, let $\triple{\alg{A}}{E}{\approx}$ be a reduced model of the logic axiomatized by these rules. Then $\alg{A}$ is a De~Morgan lattice and the rule $\ExTrue(x) \vdash x \vee y \approx x$ implies that either $E = \{ \True \}$ or $E$ is empty. The algebra  $\alg{A}$ embeds into a power of $\DMfour$ by Fact~\ref{fact: si dmas}, therefore in the former case $\triple{\alg{A}}{E}{\approx}$ embeds into a power of $\triple{\DMfour}{\{ \True \}}{=}$, and in the latter case it embeds into a power of $\triple{\DMfour}{\emptyset}{=}$. Each rule which fails in $\triple{\alg{A}}{E}{\approx}$ thus also fails either in $\triple{\DMfour}{\{ \True \}}{=}$ or in $\triple{\DMfour}{\emptyset}{=}$. But the structure $\triple{\DMfour}{\emptyset}{=}$ embeds into $\triple{\DMfour}{\{ \True \}}{=} \times \triple{\DMfour}{\{ \True \}}{=}$ via the map $x \mapsto \pair{\Neither}{x}$ or $x \mapsto \pair{\Both}{x}$.
\end{proof}

\begin{theorem}\label{thm: etl eq exp}
  In order to axiomatize the expansion of the above logic by some of the constants $\True$, $\Neither$, $\Both$, it suffices to add the appropriate subset of the following rules:
\begin{align*}
  \emptyset & \vdash \ExTrue(\True), &
  \emptyset & \vdash \ExTrue(\Neither \vee \Both), &
  \emptyset & \vdash \Both \approx \dmneg \Both, &
  \emptyset & \vdash \Neither \approx \dmneg \Neither.
\end{align*}
\end{theorem}

\begin{proof}
  The axioms allow us to exclude the case $E = \emptyset$ from consideration in the previous proof if either $\True$ or both of the constants $\Neither$, $\Both$ are present in the signature. If only $\Neither$ is present in the signature, we choose the map $x \mapsto \pair{\Neither}{x}$, and if only $\Both$ is present, we choose the map $x \mapsto \pair{\Both}{x}$ in the previous proof. (Recall that by the remarks following Fact~\ref{fact: si dmas}, the algebra $\alg{A}$ embeds into a power of $\DMfour$ even in the signature which includes one of these constants.)
\end{proof}

\begin{theorem}[Completeness theorem for the logic of $\IsTrue$, $\ExTrue$, and $\approx$]\label{thm: bde eq}~\\
  The logic of the structure $\quadruple{\DMfour}{\{ \True, \Both \}}{\{ \True \}}{=}$ in the signature $\{ \IsTrue, \ExTrue, \approx \}$ is axiomatized by extending the logic of the structure $\triple{\DMfour}{\{ \True, \Both \}}{=}$ by the rules
\begin{align*}
  \ExTrue(x) & \vdash x \vee y \approx x, & \ExTrue(x) & \vdash \IsTrue(x).
\end{align*}
\end{theorem}

\begin{proof}
  Consider a rule $\Gamma \vdash \varphi$ which fails in a non-trivial reduced model $\quadruple{\alg{A}}{T}{E}{\approx}$ of our set of rules. Then either $E = \{ \True \}$ or $E = \emptyset$. If $E =\emptyset$, then the proof of Theorem~\ref{thm: bd eq} shows that $\Gamma \vdash \varphi$ also fails in the structure $\quadruple{\DMfour}{\{ \True, \Both \}}{\emptyset}{\approx}$, which embeds into $\quadruple{\DMfour}{\{ \True, \Both\}}{\{ \True \}}{=} \times \quadruple{\DMfour}{\{ \True, \Both\}}{\{ \True \}}{=}$ via the map $x \mapsto \pair{\Neither}{x}$ or via the map $x \mapsto \pair{\Both}{x}$. The rule therefore also fails in $\quadruple{\DMfour}{\{ \True, \Both\}}{\{ \True \}}{=}$.

  On the other hand, suppose that $E = \{ \True \}$ and consider the expansion $\alg{A}'$ of $\alg{A}$ by the constant $\True$. Then the structure $\triple{\alg{A}'}{T}{\approx}$ is a model of the expansion of the logic of truth and material equivalence by the constant $\True$ axiomatized in Theorem~\ref{thm: bd eq exp}: the validity of $\ExTrue(x) \vdash x \vee y \approx x$ yields the validity of $\emptyset \vdash \True \approx \True \vee x$, and similarly the rule $\ExTrue(x) \vdash \IsTrue(x)$ yields $\emptyset \vdash \IsTrue(\True)$. Finally, the rule $\IsTrue(\dmneg \True) \vdash x \approx y$ holds in $\triple{\alg{A}'}{T}{\approx}$ because $\dmneg \True \in T$ implies $a \in T$ for each $a \in T$, therefore $a \vee b = b$ for all $a, b \in \alg{A}$ thanks to the rule $\IsTrue(x) ~ \& ~ \IsTrue(y) \vdash \dmneg x \vee y \approx y$.

  Now let $\Gamma' \vdash \varphi'$ be the rule obtained from $\Gamma \vdash \varphi$ by replacing each formula of the form $\ExTrue(u)$ by $\True \approx u$. The rule $\Gamma' \vdash \varphi'$ fails in $\triple{\alg{A}'}{T}{\approx}$, therefore it fails in $\triple{\DMfour}{\{ \True, \Both \}}{\approx}$ by Theorem~\ref{thm: bd eq exp}. Translating the rule back into its original form, it~follows that $\Gamma \vdash \varphi$ fails in $\quadruple{\DMfour}{\{ \True, \Both \}}{\{ \True \}}{\approx}$.
\end{proof}

\begin{theorem}[Completeness theorem for the logic of $\IsTrue$, $\ExTrue$, and $\approx$]\label{thm: bde eq exp}~\\
  In order to axiomatize the expansion of the above logic by some of the constants $\True$, $\Neither$, $\Both$, it suffices to add the appropriate subset of the following rules:
\begin{align*}
  \emptyset & \vdash \ExTrue(\True), & \emptyset & \vdash \Neither \approx \dmneg \Neither, & \IsTrue(\Neither \vee x) & \vdash \IsTrue(x), \\
  \emptyset & \vdash \ExTrue(\Neither \vee \Both), & \emptyset & \vdash \IsTrue(\Both \wedge \dmneg \Both), & \IsTrue(x) & \vdash \ExTrue(\Neither \vee x),
\end{align*}
\end{theorem}

\begin{proof}
  The theorem holds by the same reasoning as in Theorem~\ref{thm: bde eq}. The~additional rules ensure that the structure $\triple{\alg{A}'}{T}{\approx}$ in the proof of Theorem~\ref{thm: bde eq} is a model of the appropriate expansion of the logic of the structure $\triple{\DMfour}{\{ \True, \Both \}}{\approx}$ axiomatized by Theorems~\ref{thm: bd eq} and \ref{thm: bd eq exp}. (If $E = \emptyset$, then $\True$ is not present in the signature and neither is at least one of $\Neither$, $\Both$. We pick the map $x \mapsto \pair{\Both}{x}$ or $x \mapsto \pair{\Neither}{x}$ accordingly.)
\end{proof}

  Let us use the abbreviation $t \inequals u$ for $t \vee u \approx u$ in the following theorem.

\begin{theorem}[Completeness theorem for the logic of $\IsTrue$, $\NonFalse$, and $\approx$]\label{thm: bdnf eq}~\\
  The logic of the structure $\quadruple{\DMfour}{\!\{ \True,\!\Both \}}{\!\{ \True,\!\Neither \}}{\!=}$ in the signature $\{ \IsTrue, \NonFalse, \approx \}$ is axiomatized by the rules of $\BD$ for $\IsTrue$, the rules of $\BD$ for $\NonFalse$, the rules of De~Morgan lattices for $\approx$, and
\begin{align*}
  \IsTrue(x) ~ \& ~ \IsTrue(y) & \vdash \dmneg x \vee y \approx y, \\
  \NonFalse(x) ~ \& ~ \IsTrue(\dmneg x \vee y) & \vdash \IsTrue(y), \\
  \IsTrue(x) ~ \& ~ \NonFalse(\dmneg x \vee y) & \vdash \NonFalse(y),
\end{align*}
\begin{align*}
  \NonFalse(x) ~ \& ~ \IsTrue(y) ~ \& ~ \IsTrue(z) ~ \& ~ x \wedge y \inequals z & \vdash y \inequals z, \\
  \IsTrue(x) ~ \& ~ \NonFalse(y) ~ \& ~ x \wedge u \inequals \dmneg y \vee v ~ \& ~ x \wedge \dmneg u \inequals \dmneg y \vee v & \vdash u \inequals v.
\end{align*}
  Each reduced model of this logic has the form $\quadruple{\alg{A}}{T}{NF}{=}$, where $\alg{A}$ is a De~Morgan lattice and $T$ and $NF$ are lattice filters on $\alg{A}$ such that $T \cap NF = \{ \True \}$ or $T \cap NF = \emptyset$.
\end{theorem}

\begin{proof}
  Let $\logic{L}$ be the logic axiomatized by the given rules and $\quadruple{\alg{A}}{T}{NF}{\approx}$ be a reduced model of $\logic{L}$. Then $\alg{A}$ is a De~Morgan lattice, $T$ and $NF$ are lattice filters on~$\alg{A}$, and $\approx$ is the equality relation on $\alg{A}$ by Theorem~\ref{thm: bd eq}. Moreover, $T \cap NF$ is a singleton if non-empty because the rule $\NonFalse(x) ~ \& ~ \IsTrue(y) ~ \& ~ \IsTrue(z) ~ \& ~ x \wedge y \inequals z \vdash y \inequals z$ implies that $a \leq b$ whenever $a, b \in T \cap NF$ (take $x \assign b$, $y \assign a$, $z \assign b$).

  The structure $\quadruple{\DMfour}{\{ \True, \Both \}}{\{ \True, \Neither \}}{=}$ is a model of $\logic{L}$. Conversely, we prove that if a rule fails in a non-trivial reduced model $\quadruple{\alg{A}}{T}{NF}{\approx}$ of $\logic{L}$, then it fails in $\quadruple{\DMfour}{\{ \True, \Both \}}{\{ \True, \Neither \}}{=}$.

  Let $\set{U_{i}}{i \in I}$ be the set of all prime filters extending $T$ such that $U_{i}$ is disjoint from $\alg{A} \setminus \dmneg[NF]$. We claim that
\begin{align*}
  \quadruple{\alg{A}}{T}{NF}{\approx} = \biggl< \alg{A}, \bigcap_{i \in I} U_{i}, \bigcap_{i \in I} (\alg{A} \setminus \dmneg [U_{i}]), \bigcap_{i \in I} (\Leibniz{\alg{A}} U_{i}) \biggr>.
\end{align*}
  From this claim it follows that each rule which fails in $\quadruple{\alg{A}}{T}{NF}{\approx}$ also fails in one of the structures $\quadruple{\alg{A}}{U_{i}}{\alg{A} \setminus \dmneg[U_{i}]}{\Leibniz{\alg{A}}{U_{i}}}$. But the Leibniz reduct of such a structure is isomorphic to a substructure of $\quadruple{\DMfour}{\{ \True, \Both \}}{\{ \True, \Neither \}}{=}$ by Fact~\ref{fact: prime leibniz reduct 2}.

  In order to prove that $T = \bigcap_{i \in I} U_{i}$, it suffices to prove that for each $a \notin T$ the ideal generated by $\dmneg[NF]$ and $a$ is disjoint from $T$. This is, as in the previous proofs, ensured by the rule $\NonFalse(x), \IsTrue(\dmneg x \vee y) \vdash \IsTrue(y)$. In order to prove that $NF = \bigcap_{i \in I} \alg{A} \setminus \dmneg [U_{i}]$, it suffices to prove that $NF$ is the intersection of all prime filters $V_{i} \supseteq NF$ such that $V_{i} \cap \dmneg[T] = \emptyset$, since we can then take $U_{i} = \alg{A} \setminus \dmneg [V_{i}]$. But this is, as in the previous proofs, ensured by the rule $\IsTrue(x), \NonFalse(\dmneg x \vee y) \vdash \IsTrue(y)$.

  It now suffices to prove that for each $a < b$ in $\alg{A}$ there is some prime filter $U \supseteq T$ disjoint from $\dmneg[NF]$ such that either $b \in U$ and $a \notin U$ or $\dmneg a \in U$ and $\dmneg b \notin U$.

  Firstly, suppose that $b \in T$ and $a \notin T$. Then the filter $T$ is disjoint from the ideal generated by $\dmneg[NF]$ and $a$ thanks to the rule $\NonFalse(x) ~ \& ~ \IsTrue(\dmneg x \vee y) \vdash \IsTrue(y)$, hence there is a prime filter $U \supseteq T$ disjoint from $\dmneg[NF]$ such that $a \notin U$. We~can therefore assume that $a \in T \iff b \in T$, and likewise that $\dmneg a \in T \iff \dmneg b \in T$.

  We now construct the prime filter~$U$ case by case:

  (i) If $a, \dmneg a, b, \dmneg b \in T$, then $a = b$ because $\dmneg (\dmneg a) \leq b$ and $\dmneg (\dmneg b) \leq a$ thanks to the rule $\IsTrue(x) ~ \& ~ \IsTrue(y) \vdash \dmneg x \vee y \approx y$.

  (ii) If $a, b \in T$ and $\dmneg a, \dmneg b \notin T$, then there is some filter $T' \supseteq T$ such that $\dmneg a \in T'$ and $T'$ is disjoint from the ideal generated by $\dmneg[NF]$ and $\dmneg b$. Otherwise there are $f \in T$ and $g \in NF$ such that $f \wedge \dmneg a \leq \dmneg g \vee \dmneg b$. But $f \wedge \dmneg a = \dmneg a$ thanks to the rule $\IsTrue(x) ~ \& ~ \IsTrue(y) \vdash \dmneg x \vee y \approx y$, therefore $\dmneg a \leq \dmneg g \vee \dmneg b$ and $g \wedge b \leq a$. But then $b \leq a$ thanks to the rule $\NonFalse(x) ~ \& ~ \IsTrue(y) ~ \& ~ \IsTrue(z) ~ \& ~ x \wedge y \inequals z \vdash y \inequals z$. Extending $T'$ to a prime filter $U$ disjoint from $\dmneg [NF]$ and $b$ yields the desired filter.
   
  (iii) If $a, b \notin T$ and $\dmneg a, \dmneg b \in T$, then by (ii) with $a, b$ and $\dmneg a, \dmneg b$ exchanged, there is some prime filter $U \supseteq T$ such that $b \in U$ and $a \notin U$.

  (iv) If $a, \dmneg a, b, \dmneg b \notin T$, then there is some filter $T' \supseteq T$ such that either $b \in T'$ and $T'$ is disjoint from the ideal generated by $\dmneg[NF]$ and $a$ or $\dmneg a \in T'$ and $T'$ is disjoint from the ideal generated by $\dmneg[NF]$ and $\dmneg b$. Otherwise there are $f \in T$ and $g \in NF$ such that $f \wedge b \leq \dmneg g \vee a$ and $f \wedge \dmneg a \leq \dmneg g \vee \dmneg b$. But then $b \leq a$ by one of the rules of $\logic{L}$. Extending $T'$ to a prime filter $U$ disjoint from $\dmneg [NF]$ and $a$ (in the former case) or $\dmneg b$ (in the latter case) yields the desired filter.
\end{proof}

\begin{theorem}\label{thm: bdnf eq exp}
  In order to axiomatize the expansion of the above logic by some of the constants $\True$, $\Neither$, $\Both$, it suffices to add the appropriate subset of the following rules:
\begin{align*}
  \emptyset & \vdash \IsTrue(\True), & \emptyset & \vdash \IsTrue(\Both), & \NonFalse(\Neither), \\
  \emptyset & \vdash \NonFalse(\True), & \emptyset & \vdash \IsTrue(\dmneg \Both), & \NonFalse(\dmneg \Neither).
\end{align*}
\end{theorem}

\begin{proof}
  These rules ensure that in the previous proof the constants belong or do not belong to $U$ and $\alg{A} \setminus \dmneg [U]$ as appropriate for each constant.
\end{proof}

  Alternatively, we can exploit the fact that the set $\{ \True, \Neither \} \subseteq \DMfour$ is definable by the equation $x \wedge \Neither \approx \Neither$ in the presence of the constant $\Neither$ and the set $\{ \True, \Both \} \subseteq \DMfour$ is definable by the equation $x \wedge \Both \approx \Both$ in the presence of the constant $\Both$. In the former case, we can identify the logic of truth, non-falsity, and material equivalence with an appropriate expansion of the logic of truth and equality. In the latter case, we can identify the logic of truth, non-falsity, and material equivalence with an appropriate expansion of the logic of non-falsity and equality.

\section{Multiple-conclusion completeness}
\label{sec: mc}

  Finally, we consider the problem of axiomatizing the multiple-conclusion versions of the above logics. In multiple-conclusion logics, rules have the form $\Gamma \vdash \Delta$, where $\Gamma$~is interpreted conjunctively and $\Delta$ disjunctively. If $\Gamma = \{ \gamma_{1}, \dots, \gamma_{m} \}$ and $\Delta = \{ \delta_{1}, \dots, \delta_{n} \}$, we write $\gamma_{1} ~ \& ~ \dots ~ \& ~ \gamma_{m} \vdash \delta_{1} \text{ or } \dots \text{ or } \delta_{n}$. The semantics of single-conclusion logics extends straightforwardly to such rules. The rule
\begin{align*}
  \NonFalse(x \vee y) \vdash \NonFalse(x) \text{ or } \NonFalse(y)
\end{align*}
  expresses the fact that if $x \vee y$ is non-false, then either $x$ is non-false or $y$ is non-false. Similarly, the rule
\begin{align*}
  \IsTrue(\dmneg x) ~ \& ~ \NonFalse(x) & \vdash \emptyset
\end{align*}
  expresses the fact that if $x$ is non-false, then $\dmneg x$ is not true. Just like single-conclusion rules correspond to strict universal Horn sentences (without equality), multiple-conclusion rules correspond to universal sentences (without equality), at least if $\Gamma$ and $\Delta$ are finite. That is, each multiple-conclusion rule $\Gamma \vdash \Delta$ corresponds to a universally quantified disjunction of atomic formulas and negated atomic formulas. For example, the rule $\NonFalse(x \vee y) \vdash \NonFalse(x) \text{ or } \NonFalse(y)$ corresponds to the universal sentence
\begin{align*}
  \forall x \, \forall y \, (\neg \NonFalse(x \vee y) \vee \NonFalse(x) \vee \NonFalse(y)).
\end{align*}

  Axiomatizing the multiple-conclusion versions of most of the logics studied above turns out to be a trivial task. If the relational signature contains the equality symbol, it is trivial because the property of being isomorphic to a substructure of a given finite structure is axiomatizable by universal sentences with equality.

  If the equality symbol is not part of the signature, but at least one of the predicates $\IsTrue$ or $\NonFalse$ is, say $\IsTrue$, then the task is trivial because the other predicates are definable in terms of this predicate:
\begin{align*}
  a \in \{ \True \} & \iff a \in \{ \True, \Both \} ~ \& ~ a \notin \{ \True, \Both \}, &
  a \in \{ \True, \Neither \} & \iff \dmneg a \notin \{ \True, \Both \}.
\end{align*}
  In a multiple-conclusion setting, these equivalences can be expressed by the rules
\begin{gather*}
  \ExTrue(x) \vdash \IsTrue(x), \\
  \IsTrue(\dmneg x) ~ \& ~ \ExTrue(x) \vdash \emptyset, \\
  \IsTrue(x) \vdash \ExTrue(x) \text{ or } \IsTrue(\dmneg x),
\end{gather*}
\begin{align*}
  \emptyset \vdash \IsTrue(x) \text{ or } \NonFalse(\dmneg x), \\
  \IsTrue(x) ~ \& ~ \NonFalse(\dmneg x) \vdash \emptyset.
\end{align*}
  In a structure which satisfies these rules, each multiple-conclusion rule is in fact equivalent to a multiple-conclusion rule in the signature $\{ \IsTrue \}$, meaning that one rule holds if and only if the other does. But~we know how to axiomatize the multiple-conclusion version of Belnap--Dunn logic: it suffices to add the rule $\IsTrue(x \vee y) \vdash \IsTrue(x) \text{ or } \IsTrue(y)$ to a single-conclusion axiomatization.

  The only relational signature for which the problem of axiomatizing the multiple-conclusion logic is non-trivial is therefore $\{ \ExTrue \}$. An axiomatization for the multiple-conclusion version of Exactly True Logic, possibly extended by some of the constants $\True$, $\Neither$, $\Both$, is thus given below. This completes our review of relational variants of the four-valued logic of Belnap and Dunn.

\begin{theorem}[Completeness theorem for the multiple-conclusion logic of $\ExTrue$]\label{thm: etl mc}
  The multiple-conclusion logic of the structure $\pair{\DMfour}{\{ \True \}}$ in the signature $\{ \ExTrue \}$ is axiomatized by the extending an axiomatization of the single-conclusion logic of this structure by the rules
\begin{align*}
  x \vee y & \vdash \dmneg x \vee \dmneg y ~\mathrm{or}~ x ~\mathrm{or}~y, \\
  x \vee y & \vdash \dmneg x \vee y ~\mathrm{or}~ x ~\mathrm{or}~y,
\end{align*}
\begin{align*}
  (u \wedge \dmneg u) \vee x ~ \& ~ (u \wedge \dmneg u) \vee y ~ \& ~ v \vee x & \vdash v \vee y~\mathrm{or}~x~\mathrm{or}~y, \\
  (u \wedge \dmneg u) \vee x ~ \& ~ (u \wedge \dmneg u) \vee y ~ \& ~ v \vee \dmneg x & \vdash v \vee \dmneg y~\mathrm{or}~x~\mathrm{or}~y,
\end{align*}
  where for reasons of space we omit the predicate $\ExTrue$ around each term.
\end{theorem}

\begin{proof}
  These rules are satisfied by $\pair{\DMfour}{\{ \True \}}$. Conversely, suppose that a rule fails in a non-trivial reduced model $\pair{\alg{A}}{E}$ of these rules. We know that $E = \{ \True \}$ or $E = \emptyset$ by Fact~\ref{fact: models of bd}, and $E = \emptyset$ only if $\alg{A}$ is a singleton. If $E = \emptyset$, then $\pair{\alg{A}}{E}$ validates each rule with a non-empty set of premises, therefore the rule in question must have an empty set of premises. But each such rule fails in $\pair{\DMfour}{\{ \True \}}$ if we interpret each variable by $\Neither$ (or by $\Both$). We may therefore assume that $E = \{ \True \}$.

  If there are no $a, b < \True$ such that $a \vee b = \True$, then the equivalence relation which collapses all elements of $\alg{A}$ other than $\True$~and~$\False$ is a compatible congruence and the structure $\pair{\alg{A}}{E}$ is isomorphic to either $\pair{\Btwo}{\{ \True \}}$ or $\pair{\Kthree}{\{ \True \}}$. But these are both sub\-structures of $\pair{\DMfour}{\{ \True \}}$, therefore each rule which fails in one of the also fails in the structure $\pair{\DMfour}{\{ \True \}}$.

  Suppose therefore that there are $a, b < \True$ such that $a \vee b = \True$. The first two rules in our axiomatization ensure that without loss of generality we can take $a \leq \dmneg a$ and $b \leq \dmneg b$. Because the rule
\begin{align*}
  x \vee y \approx \True ~ \& ~ x \approx x \wedge \dmneg x ~ \& ~ y \approx y \wedge \dmneg y & \vdash x \approx x \vee \dmneg x
\end{align*}
  holds in $\DMfour$, it also holds in the De~Morgan lattice $\alg{A}$, hence $a = \dmneg a$ and $b = \dmneg b$. We now claim that the following implications hold:
\begin{align*}
  a \vee c = \True \text{ for } \False < c < \True & \implies c = b, \\
  b \vee c = \True \text{ for } \False < c < \True & \implies c = a, \\
  a \vee c < \True \text{ for } \False < c < \True & \implies c = a, \\
  b \vee c < \True \text{ for } \False < c < \True & \implies c = b.
\end{align*}
  Consider for example the first of these implications. It can be reformulated as
\begin{align*}
  \ExTrue(a \vee c) & \implies \ExTrue(c) \text{ or } \ExTrue(\dmneg c) \text{ or } \pair{c}{b} \in \Leibniz{\alg{A}} E.
\end{align*}
  But this is precisely what the last two axioms express if we recall the description of Leibniz congruences of models of $\BD$ (Fact~\ref{fact: leibniz of bd}). The four implications above now imply that $\alg{A}$ has exactly four distinct elements, namely $\False, \True, a, b$. We also know that $a \vee b = \True$ and $a = \dmneg a$ and $b = \dmneg b$. But this implies that $\alg{A}$ is isomorphic to $\DMfour$ and $\pair{\alg{A}}{E}$ is isomorphic to $\pair{\DMfour}{\{ \True \}}$.
\end{proof}

\begin{theorem}
  In order to axiomatize the expansion of the above logic by some of the constants $\True$, $\Neither$, $\Both$, it suffices to add the appropriate subset of the following rules:
\begin{align*}
  \emptyset & \vdash \ExTrue(\True), & \ExTrue(\Neither) & \vdash \emptyset, & \ExTrue(\Both) & \vdash \emptyset, \\
  \emptyset & \vdash \ExTrue(\Neither \vee \Both), & \ExTrue(\dmneg \Neither) & \vdash \emptyset, & \ExTrue(\dmneg \Both) & \vdash \emptyset.
\end{align*}
\end{theorem}

\begin{proof}
  These rules ensure that in each reduced model $\True$ is the top element, that $\Neither$ is one of the elements $a$ or $b$ in the previous proof (or the middle element of $\Kthree$), and likewise for $\Both$. Moreover, if both $\Neither$ and $\Both$ are in the signature, then they are Boolean complements in the distributive lattice reduct, thus if $\Neither$ is $a$ ($b$), then $\Both$ is $b$ ($a$), and vice versa. The structure $\pair{\alg{A}}{E}$ is therefore isomorphic to a substructure of the appropriate expansion of $\pair{\DMfour}{\{ \True \}}$.
\end{proof}

\end{document}